\documentclass[11pt,reqno]{amsart}
\usepackage{amsmath}
\usepackage{amssymb}
\usepackage{amstext}
\usepackage{amsfonts}
\usepackage{amsthm}
\usepackage{ifthen}
\usepackage{xspace}
\usepackage{graphicx}
\usepackage{mathrsfs}
\usepackage{enumerate}
\usepackage{breakcites}

\numberwithin{equation}{section}  

\newtheorem{punkt}{}[section]

\theoremstyle{plain}

\newtheorem{lemma}[punkt]{Lemma}

\newtheorem{theorem}[punkt]{Theorem}
\newtheorem{conjecture}[punkt]{Conjecture}

\theoremstyle{definition}

\theoremstyle{plain}
\newtheorem*{corollary*}{Corollary}
\newtheorem*{lemma*}{Lemma}
\newtheorem*{proposition*}{Proposition}
\newtheorem*{theorem*}{Theorem}
\newtheorem*{conjecture*}{Conjecture}

\theoremstyle{definition}
\newtheorem*{remark*}{Remark}
\newtheorem*{remarks*}{Remarks}
\newtheorem*{example*}{Example}
\newtheorem*{examples*}{Examples}
\newtheorem*{definition*}{Definition}
\newtheorem*{assumption*}{Assumption}
\newtheorem*{assumptions*}{Assumptions}
\newtheorem*{construction*}{Construction}

\def\mylebesgue{\lambda \mskip -8mu \lambda}

\def\myo{\mathcal{O}}
\def\myp{\mathcal{P}}

\def\mys{\mathcal{S}}

\def\re{\qopname\relax{no}{Re}\,}

\def\eg{e.g.\@\xspace}
\def\ie{i.e.\@\xspace}

\begin{document}


\title{On the Occurrence of the Sine Kernel \\[+5pt] in Connection with the Shifted Moments \\[+5pt] of the Riemann zeta function}

\author{H. K\"osters}
\address{Holger K\"osters, Fakult\"at f\"ur Mathematik, Universit\"at Bielefeld,
Postfach 100131, 33501 Bielefeld, Germany}
\email{hkoesters@math.uni-bielefeld.de}

\date{August 31, 2008}

\begin{abstract}
We point out an interesting occurrence of the sine kernel
in connection with the shifted moments 
of the Riemann zeta function along the critical line. 
We discuss rigorous results in this direction
for the shifted second moment and for the shifted fourth moment.
%
%
Furthermore, we conjecture that the sine kernel
also occurs in connection with the higher (even) shifted moments
and show that this conjecture is closely related to a recent conjecture
by \textsc{Conrey}, \textsc{Farmer}, \textsc{Keating},
\textsc{Rubinstein}, and \textsc{Snaith} \cite{CFKRS1,CFKRS2}.
\end{abstract}


\maketitle

\markboth{H. K\"osters}{The Riemann Zeta Function and The Sine Kernel}

\bigskip

\def\fc{f_{\text{CUE}}}
\def\fg{f_{\text{GUE}}}

\section{Introduction}

Since the discovery by Montgomery and Dyson that
the pair correlation function of the non-trivial zeros 
of the Riemann zeta function seems to be asymptotically
the~same~as~that of the eigenvalues of a random matrix 
from the Gaussian Unitary Ensemble (GUE),
the relationship between the theory of the Riemann zeta function and 
the theory of random matrices has attracted considerable interest.
This interest intensified in the last few years
after \textsc{Keating} and \textsc{Snaith} \cite{KS1}
compared the~moments of the characteristic polynomial
of a random matrix from the Circular Unitary Ensemble (CUE)
with the -- partly conjectural -- moments of the~value distribution 
of the Riemann zeta function along the critical line, 
and also found some striking similarities.
These findings have sparked intensive further research.
On~the one hand, there are now~a~number of new conjectures,
derived from random matrix theory, 
about the moments~of the value distribution 
of the Riemann zeta function and more general $L$-functions
(see the papers by \textsc{Keating} and \textsc{Snaith} \cite{KS1,KS2}
as well as \textsc{Conrey}, \textsc{Farmer}, \textsc{Keating}, 
\textsc{Rubinstein}, and \textsc{Snaith} \cite{CFKRS1,CFKRS2,CFKRS3}
and the references contained therein).
On the other hand, various authors have investigated 
the moments and the correlation functions of the characteristic polynomial
also for other random matrix ensembles
(see \eg \textsc{Br\'ezin} and \textsc{Hikami} \cite{BH1,BH2},
\textsc{Mehta} and \textsc{Normand} \cite{MN},
\textsc{Strahov} and \textsc{Fyodorov} \cite{SF3},
\textsc{G\"otze} and \textsc{K\"osters} \cite{GK}).

A recurring phenomenon on the random matrix side 
is the emergence of the sine~kernel in the asymptotics
of the correlation functions (or shifted moments)
of the characteristic polynomial.
For instance, for the Circular Unitary Ensemble (CUE)
(see \textsc{Forrester} \cite{Fo} or \textsc{Mehta} \cite{Me}),
the second-order correlation function of the characteristic polynomial
$$
   \fc(N,\mu,\nu) 
:= \int_{\mathcal{U}_N} \det(U - \mu I) \, \overline{\det(U - \nu I)} \ dU
$$
(where $I$ denotes the $N \times N$ identity matrix
and integration is with respect to the normalized Haar measure
on the group $\mathcal{U}_N$ of $N \times N$ unitary matrices)
satisfies
\begin{align}
\label{cue1}
\lim_{N \to \infty} \frac{1}{N} \cdot \fc \left( N; e^{2\pi i\mu/N}, e^{2\pi i\nu/N} \right) 
= 
e^{\pi i(\mu-\nu)} \cdot \frac{\sin \pi(\mu-\nu)}{\pi(\mu-\nu)}
\end{align}
for any $\mu,\nu \in \mathbb{R}$.
This can be deduced using standard arguments from random matrix theory
(see \eg Chapter 4 in \textsc{Forrester} \cite{Fo}).
More generally, using similar arguments, it can be shown that for any $M \geq 1$,
the correlation function of order $2M$ of the characteristic polynomial
$$
   \fc(N,\mu_1,\hdots,\mu_M,\nu_1,\hdots,\nu_M) 
:= \int_{\mathcal{U}_N} \prod_{j=1}^{M} \det(U - \mu_j I) \, \overline{\det(U - \nu_j I)} \ dU
$$
satisfies
\begin{multline}
\label{cue2}
\lim_{N \to \infty} \frac{1}{N^{M^2}} \cdot \fc \left( N; e^{2\pi i\mu_1/N},\hdots,e^{2\pi i\mu_M/N}, e^{2\pi i\nu_1/N},\hdots,e^{2\pi i\nu_M/N} \right) 
\\ = 
\frac{\exp(\sum_{j=1}^{M} \pi i(\mu_j-\nu_j))}{\Delta(2\pi\mu_1,...,2\pi\mu_M) \cdot {\Delta(2\pi\nu_1,...,2\pi\nu_M)}} \cdot \det \left( \frac{\sin \pi(\mu_j-\nu_k)}{\pi(\mu_j-\nu_k)} \right)_{j,k=1,\hdots,M}
\end{multline}
for any pairwise different $\mu_1,\hdots,\mu_M,\nu_1,\hdots,\nu_M \in \mathbb{R}$,
where $\Delta(x_1,\hdots,x_M) := \prod_{1 \leq j<k \leq M} (x_k - x_j)$
denotes the Vandermonde determinant.
For completeness, the proof of (\ref{cue2})
is sketched in Appendix~B of this paper.

Similarly, for the Gaussian Unitary Ensemble (GUE)
(see \textsc{Forrester} \cite{Fo} or \textsc{Mehta} \cite{Me}),
the second-order correlation function of the characteristic polynomial
$$
   \fg(N,\mu,\nu) 
:= \int_{\mathcal{H}_N} \det(X - \mu I) \, \det(X - \nu I) \ \mathbb{Q}(dX)
$$
(where $I$ denotes the $N \times N$ identity matrix
and $\mathbb{Q}$ denotes the Gaussian Unitary Ensemble on the space 
$\mathcal{H}_N$ of $N \times N$ Hermitian matrices)
satisfies
\begin{align}
\label{gue1}
\lim_{N \to \infty} \sqrt{\frac{\pi}{2N}} \cdot \frac{2^N}{N!} \cdot \fg \left( N; \frac{\pi\mu}{\sqrt{2N}}, \frac{\pi\nu}{\sqrt{2N}} \right) 
= 
\frac{\sin \pi(\mu-\nu)}{\pi(\mu-\nu)}
\end{align}
for any $\mu,\nu \in \mathbb{R}$
(see \eg Chapter 4 in \textsc{Forrester} \cite{Fo}).
Also, an analogue of~(\ref{cue2}) holds as~well.
Even more, these results can be generalized both 
to the class of unitary-invariant matrix ensembles
(\textsc{Br\'ezin} and \textsc{Hikami} \cite{BH1},
\textsc{Mehta} and \textsc{Normand} \cite{MN},
\textsc{Strahov} and \textsc{Fyodorov} \cite{SF3})
and -- at least for the second-order correlation function --
to the class of Hermitian Wigner ensembles
(\textsc{G\"otze}~and~\textsc{K\"osters} \cite{GK}).
In particular, it is noteworthy 
that the emergence of the sine kernel is universal,
as it occurs in all the cases pre\-viously mentioned,
irrespective of the particular details of the~definition 
of the random matrix ensemble.
(More~precisely, the emergence of the sine kernel
depends on the symmetry class of the random matrix ensemble. 
For instance, for the Gaussian Orthogonal Ensemble (GOE)
on the space of real symmetric matrices,
the asymptotics are different;
see \textsc{Br\'ezin} and \textsc{Hikami} \cite{BH2}.)

In view of the above-mentioned similarities
between random matrices and the Riemann zeta function,
it seems natural to ask whether there is an analogue 
of (\ref{cue1})~and~(\ref{cue2}) for the shifted moments 
of the Riemann zeta function along the critical line.
Although we have the feeling that such analogues should be \mbox{well-known},
we have not been able to find an explicit statement of such analogues
in the (extensive!) recent literature on the relationship between 
the theory of random matrices and the theory of the Riemann zeta function.
This seems somewhat surprising, particularly since 
there exist some more general results (or conjectures)
from which an analogue of (\ref{cue1})~or~(\ref{cue2}) 
could be deduced rather easily, at least on a formal level.
The main aim of this note is to fill this gap.

More precisely, by an analogue of (\ref{cue1}) and (\ref{gue1}), 
we mean a result of the form
\begin{multline*}
\lim_{T \to \infty} \frac{1}{C(T)} \int_{T_0}^{T} \zeta \Big( \tfrac{1}{2} + i \Big( t + \tfrac{2\pi\mu}{\log t} \Big) \Big) \, \zeta \Big( \tfrac{1}{2} - i \Big( t + \tfrac{2\pi\nu}{\log t} \Big) \Big) \ dt
=
e^{\pm i\pi(\mu-\nu)} \cdot \frac{\sin \pi(\mu-\nu)}{\pi(\mu-\nu)} \,,
\end{multline*}
where $\mu$ and $\nu$ are arbitrary real numbers, $T_0 > 1$ is a constant,
and $C(T)$ is some normalizing factor depending on $T$.
To account for our choice of scaling for the shift parameters $\mu$~and~$\nu$,
note that both in (\ref{cue1}) and in (\ref{gue1}),
the scaling factor is equal to the mean spacing of eigenvalues.
For~instance, for a random $N \times N$ matrix from the~CUE,
there~are $N$ eigenvalues distributed over the unit circle 
of length~$2\pi$, which gives rise to a~mean spacing of $2\pi/N$.
Similarly, for a random $N \times N$ matrix from the~GUE, 
it~is well-known that the mean spacing at the origin is $\pi/\sqrt{2N}$ 
(see \eg Chapter~6 in \textsc{Mehta}~\cite{Me}). 
Now recall that, if $N(T)$ denotes the number of zeros of $\zeta(\sigma + it)$ 
in the region $0 \leq \sigma \leq 1$, $0 \leq t \leq T$,
it is known that $N(T) \sim (2\pi)^{-1} \, T \, \log T$
(see \eg Chapter~9 in \textsc{Titchmarsh}~\cite{Ti}), so that 
the \emph{empirical} mean spacing at location $t$ is $\sim 2\pi/\log t$.
Since this mean spacing depends on $t$,
it seems natural to multiply the shift parameters $\mu$~and~$\nu$
by~the \emph{location-dependent} scaling factor $2\pi/\log t$.

For the shifted second moment of the Riemann zeta function,
this result was obtained (in a slightly different formulation)
already by \textsc{Atkinson} \cite{At} in 1948,
and thus even before the sine kernel was ``discovered''
in random matrix theory. (Curiously, this paper seems 
not to get cited in the recent literature on the interplay 
between random matrix theory and number theory.)
\textsc{Atkinson}'s theorem can be restated as follows:

\begin{theorem}
\label{zeta2theorem}
For any $T_0 > 1$ and any $\mu,\nu \in \mathbb{R}$,
$$ 
\lim_{T \to \infty} \frac{1}{T \log T} \int_{T_0}^{T} \zeta \Big( \tfrac{1}{2} + i \Big( t + \tfrac{2\pi\mu}{\log t} \Big) \Big) \, \zeta \Big( \tfrac{1}{2} - i \Big( t + \tfrac{2\pi\nu}{\log t} \Big) \Big) \ dt
\\ =
e^{-i\pi(\mu-\nu)} \cdot \mathbb{S}(\pi(\mu-\nu)) \,,
$$ 
where $\mathbb{S}(x) := \sin x / x$ for $x \ne 0$
and $\mathbb{S}(x) := 1$ for $x = 0$.
\end{theorem}

In particular, for $\mu,\nu = 0$, this reduces to the classical result that
$$
\lim_{T \to \infty} \frac{1}{T \log T} \int_{0}^{T} \left| \zeta \Big( \tfrac{1}{2} + i t \Big) \right|^2 \ dt
= 
1
$$
(see \eg Theorem~7.3 in \textsc{Titchmarsh} \cite{Ti}).

\pagebreak[3]

Actually, \textsc{Atkinson}'s theorem states that for any $\alpha \geq 0$,
$$
\int_{T_0}^{T} \zeta \Big( \tfrac{1}{2} + i u(t) \Big) \, \zeta \Big( \tfrac{1}{2} - i t \Big) \ dt
\sim
e^{-i u} \cdot \mathbb{S}(u) \cdot T \log T 
\qquad
(T \to \infty)\,,
$$
where $u(t)$ is defined by the relation $\vartheta(u(t)) - \vartheta(t) \!=\! \alpha$,
with $\vartheta(t) \!:=\! - \tfrac{1}{2} \arg \chi (\tfrac{1}{2} + it)$.
However, as $u(t) - t \sim 2\alpha / \log t$ ($t \to \infty$),
it seems clear that Theorem \ref{zeta2theorem} is virtually the~same,
and in fact this result can be established by the same proof
as \textsc{Atkinson}'s theorem. 

For the shifted fourth moment of the Riemann zeta function,
we have the following result, which constitutes
an analogue of (\ref{cue2}) in the special case
$M := 2$, $\mu_1 = \nu_1 =: \mu$, $\mu_2 = \nu_2 =: \nu$:

\begin{theorem}
\label{zeta4theorem}
For any $T_0 > 1$ and any $\mu,\nu \in \mathbb{R}$,
$$
\lim_{T \to \infty} \frac{1}{T (\log T)^4} \int_{T_0}^{T} \Big| \zeta \Big( \tfrac{1}{2} + i \Big( t + \tfrac{2\pi\mu}{\log t} \Big) \Big) \Big|^2 \, \Big| \zeta \Big( \tfrac{1}{2} + i \Big( t + \tfrac{2\pi\nu}{\log t} \Big) \Big) \Big|^2 \ dt
\\ =
\frac{3}{2\pi^2} \cdot \mathbb{T}(\pi(\mu-\nu)) ,
$$
where $\mathbb{T}(x) := \frac{1}{x^2} \left( 1 - \left( \frac{\sin x}{x} \right)^2 \right)$
for $x \ne 0$ and $\mathbb{T}(0) := 1/3$ for $x = 0$.
\end{theorem}

In particular, for $\mu,\nu = 0$, this reduces to the classical result that
$$
\lim_{T \to \infty} \frac{1}{T (\log T)^4} \int_{0}^{T} \Big| \zeta \Big( \tfrac{1}{2} + it \Big) \Big|^4 \ dt
\\ =
\frac{1}{2\pi^2}
$$
(see Theorem~B in \textsc{Ingham} \cite{In}).

For the proof of Theorem \ref{zeta4theorem}, we will closely follow 
the proof of Theorem~B in \textsc{Ingham} \cite{In}. In particular, 
our proof is also based on the approximate functional equation for $\zeta(s)^2$.
(This is analogous to the proof of Theorem \ref{zeta2theorem}
indicated above, which closely follows the proof
of the corresponding result for the non-shifted second moment,
starting from the approximate functional equation for $\zeta(s)$.)

As pointed out by an anonymous referee, it should also be possible
to deduce Theorems \ref{zeta2theorem} and \ref{zeta4theorem}
(and even more precise versions involving information about the lower-order terms)
from the existing (more general) mean value theorems 
for the second and fourth moment of the Riemann zeta~function 
with \emph{constant} shifts
(see Theorem~A in \textsc{Ingham} \cite{In}
and Theorem~4.2 in \textsc{Motohashi} \cite{Mot}).
However, we will not pursue this issue further here,
since it is our main aim to point out that the highest-order terms
of the appropriately shifted moments of the Riemann zeta~function
give rise to the sine kernel.
Furthermore, weighing the shifts with the~factor $2\pi / \log t$
seems to simplify the problem, and we therefore think 
that our~independent (and comparatively simple) proof
of Theorem \ref{zeta4theorem} might be of its~own~interest.

As regards the higher (even) shifted moments 
of the Riemann zeta function along the critical line,
we will show that a recent conjecture
by \textsc{Conrey}, \textsc{Farmer}, \textsc{Keating}, 
\textsc{Rubinstein}, and \textsc{Snaith} \cite{CFKRS2},
when combined with our choice of scaling,
gives rise to the following analogue of (\ref{cue2}):

\begin{conjecture}
\label{skc}
For any $M=1,2,3,\hdots$, for any $T_0 > 1$ and for any $\mu_1,\hdots,\mu_M, \linebreak[2] \nu_1,\hdots,\nu_M \in \mathbb{R}$,
\begin{multline*}
\lim_{T \to \infty}
\frac{1}{T (\log T)^{M^2}}
\int_{T_0}^{T} 
\prod_{j=1}^{M} \zeta \left(\tfrac{1}{2} + it + \tfrac{2\pi i\mu_j}{\log t} \right)
\prod_{j=1}^{M} \zeta \left(\tfrac{1}{2} - it - \tfrac{2\pi i\nu_j}{\log t} \right)
\ dt
\\
= a_M \cdot \frac{\exp ( -\pi i \sum_{j=1}^{M} (\mu_j - \nu_j) )}{\Delta(2\pi\mu_1,...,2\pi\mu_M) \cdot {\Delta(2\pi\nu_1,...,2\pi\nu_M)}} \cdot \det \left( \frac{\sin \pi(\mu_j-\nu_k)}{\pi(\mu_j-\nu_k)} \right)_{j,k=1,\hdots,M} \,,
\end{multline*}
where $\Delta(x_1,\hdots,x_M) := \prod_{1 \le j < k \le M} (x_k - x_j)$
is the Vandermonde determinant and
$$
a_M := \prod_{p \in \myp} \left( 1 - \tfrac{1}{p} \right)^{M^2} \, \sum_{j=0}^{\infty} \left( \frac{\Gamma(j+M)}{j! \, \Gamma(M)} \right)^2 p^{-j} \,,
$$
the product being taken over the set $\myp$ of prime numbers.
(Naturally, in the case where two or more of the shift parameters are equal,
the right-hand side should be regarded as defined by continuous extension, 
similarly as in the preceding theorems.)
\end{conjecture}

It is easy to see that $a_1 = 1$ and $a_2 = 6/\pi^2$. 
Thus, Theorem~\ref{zeta2theorem} confirms Conjecture~\ref{skc}
in the special case $M = 1$,
and Theorem~\ref{zeta4theorem} confirms Conjecture~\ref{skc}
in the special case $M = 2$, $\mu_1 = \nu_1$, $\mu_2 = \nu_2$.

Furthermore, Equation (\ref{cue2}) and Conjecture \ref{skc}
clearly have a similar structure.
A notable difference is given by the factor $a_M$
which occurs in~Conjecture \ref{skc} for the Riemann zeta function
but not in~Equation (\ref{cue2}) for the CUE.
It is well-known \linebreak
(see \eg \textsc{Keating} and \textsc{Snaith} \cite{KS1,KS2}
and \textsc{Conrey}, \textsc{Farmer}, \textsc{Keating},
\textsc{Rubin\-stein}, and \textsc{Snaith} \cite{CFKRS1,CFKRS2})
that this ``arithmetic factor'' is not~predicted by random matrix theory.
Another difference is given by the sign in the phase factor
$\exp ( \pm \pi i \sum_{j=1}^{M} (\mu_j - \nu_j) )$.
This difference could have been avoided if we had defined 
the characteristic polynomial by \mbox{$\det (I - \xi^{-1} U)$} 
instead of \mbox{$\det (U - \xi I)$}.

As already mentioned, apart from \textsc{Atkinson}'s theorem,
we are not aware of explicit statements of \emph{continuous} 
mean value theorems involving the sine kernel.
This seems somewhat surprising, given that
the choice of the shift parameters on the scale $2\pi / \log t$
seems completely natural in view of the existing similarities
to random matrix theory.
In particular, this scaling also occurs 
in the pair correlation function 
of the zeros of the Riemann zeta function
(see \eg \textsc{Montgomery} \cite{Mon})
as well as in a number of \emph{discrete} mean value theorems
related to the zeros of the Riemann zeta function
(see \eg \textsc{Gonek}~\cite{Go}, \textsc{Hughes} \cite{Hu},
\textsc{Mozer} \cite{Moz1,Moz2,Moz3}).

Throughout this paper, we use the following notation:
Let $\zeta(s)$ denote the Riemann zeta function,
which is defined by the Dirichlet series
$$
\zeta(s) := \sum_{n=1}^{\infty} n^{-s}
$$
for $\re(s) > 1$ and by analytic continuation
for $\re(s) \leq 1$, and let 
$$
\chi(s) := 2^s \, \pi^{s-1} \, \sin (\tfrac{1}{2} \pi s) \, \Gamma(1-s) = \pi^{s-\frac{1}{2}} \, \Gamma(\tfrac{1}{2} - \tfrac{1}{2}s) \,/\, \Gamma(\tfrac{1}{2}s)
$$
for any $s \in \mathbb{C}$. \pagebreak[1]
We follow the convention of denoting 
the real and imaginary~part of the argument $s$ 
by $\sigma$ and $t$, respectively.
Furthermore, for any integer $n \geq 1$,
we~denote by $d(n)$ the number of divisors of $n$.
Finally, we make the convention 
that, unless otherwise indicated, the $\myo$-bounds 
occurring in the proofs may depend on $\mu$ and $\nu$ 
(which are regarded as fixed) but not on any other parameters.

This paper is structured as follows.
Section~2 is devoted to the proof of \linebreak Theorem~\ref{zeta4theorem}.
In Section~3 we discuss the relationship between Conjecture~\ref{skc}
for the higher (even) shifted moments of the Riemann zeta function
and the conjecture by \textsc{Conrey}, \textsc{Farmer}, \textsc{Keating}, 
\textsc{Rubinstein}, and \textsc{Snaith} \pagebreak[2] \cite{CFKRS1,CFKRS2}.
Finally, for the convenience of the reader, the appendices A~and~B
contain some auxiliary results from analytic number theory
and random matrix theory which have been used in the preceding sections.

\bigskip

\section{The Mean Value of the Fourth Moment}

The proof of Theorem~\ref{zeta4theorem} is modelled on the proof
of Theorem~B in \textsc{Ingham}~\cite{In}).

\begin{proof}[Proof of Theorem \ref{zeta4theorem}]
We will show that for any $\mu,\nu \in \mathbb{R}$,
\begin{multline}
\label{zeta4claim}
\lim_{T \to \infty} \frac{1}{T (\log T)^4} \int_{T}^{2T} \Big| \zeta \Big( \tfrac{1}{2} + i \Big( t + \tfrac{2\pi\mu}{\log t} \Big) \Big) \Big|^2 \, \Big| \zeta \Big( \tfrac{1}{2} + i \Big( t + \tfrac{2\pi\nu}{\log t} \Big) \Big) \Big|^2 \ dt
\\ =
\frac{3}{2\pi^4(\mu-\nu)^2} \cdot \left( 1 - \left( \frac{\sin \pi(\mu-\nu)}{\pi(\mu-\nu)} \right)^2 \right) \,.
\end{multline}
The assertion of Theorem \ref{zeta4theorem} then follows
by using (\ref{zeta4claim}) for $T/2^1$, $T/2^2$, $T/2^3$, \dots{}
and taking the~sum.

For the proof of (\ref{zeta4claim}),
we start from the approximate functional equation 
for $\zeta^2$ (see~\eg Theorem 4.2 in \textsc{Ivi\'c} \cite{Iv}) ,
which states that for any $h > 0$,
$$
\zeta^2(s) = \sum_{n \leq x} d(n) n^{-s} + \chi^2(s) \sum_{n \leq y} d(n) n^{-1+s} + \myo(x^{1/2-\sigma} \log t)
$$
for $0 < \sigma < 1$, $4\pi^2xy = t^2$, $x,y,t > h > 0$.
Taking $\sigma = \tfrac{1}{2}$, $t > 2$, $x(t) = y(t) = t/2\pi$,
it follows that
$$
\zeta^2(\tfrac{1}{2} + it) = \sum_{n \leq x(t)} d(n) \, n^{-\tfrac{1}{2}-it} + \chi^2(\tfrac{1}{2}+it) \sum_{n \leq x(t)} d(n) \, n^{-\tfrac{1}{2}+it} + \myo(\log t) \,.
$$
Using the functional equation
$$
  \zeta(\tfrac{1}{2} + it) 
= \chi(\tfrac{1}{2} + it) \, \zeta(\tfrac{1}{2} - it)
= \chi(\tfrac{1}{2} + it) \, \overline{\zeta(\tfrac{1}{2} + it)}
$$
(see \eg Equation (1.23) in \textsc{Ivi\'c} \cite{Iv})
and multiplying by $\chi(\tfrac{1}{2} + it)^{-1} = \overline{\chi(\tfrac{1}{2} + it)}$,
which is of order $\myo(1)$, we there\-fore obtain
$$
\big| \zeta(\tfrac{1}{2} + it) \big|^2 = 2 \re \left( \chi(\tfrac{1}{2}+it) \sum_{n \leq x(t)} d(n) \, n^{-\tfrac{1}{2}+it} \right) + \myo(\log t) \,.
$$
(This is Equation (4.11) in \textsc{Ivi\'c} \cite{Iv}.)

\pagebreak[2]

In the following, we will repeatedly use the fact
that for any fixed $\varepsilon > 0$,
$$
d(n) = \myo(n^{\varepsilon})
$$
(see \eg Equation (1.71) in \textsc{Ivi\'c}~\cite{Iv}).
In particular, this implies that
\begin{align*}
  \sum_{n \leq x(t)} d(n) n^{-\frac{1}{2}}
= \myo\bigg( \sum_{n \leq x(t)} n^{-\frac{1}{2}+\varepsilon} \bigg)
= \myo\bigg( \int_1^{x(t)} u^{-\frac{1}{2}+\varepsilon} \ du \bigg)
= \myo( t^{\frac{1}{2}+\varepsilon} ) \,.
\end{align*}
Using the approximation
$$
\chi(\tfrac{1}{2} + it) = e^{\pi i/4} \left(\frac{2\pi e}{t}\right)^{it} + \myo(t^{-1}) \qquad (t > 2)
$$
(see \eg Equation (1.25) in \textsc{Ivi\'c} \cite{Iv}), it follows that
$$
\big| \zeta(\tfrac{1}{2} + it) \big|^2 = 2 \re \left( e^{\pi i/4} \left(\frac{2\pi e}{t}\right)^{it} \sum_{n \leq x(t)} d(n) \, n^{-\tfrac{1}{2}+it} \right) + \myo(\log t) \,.
$$
Using this equation for $t + 2\pi\lambda / \log t$ instead of $t$,
where $\lambda $ is a fixed real number and $t$ is sufficiently large
(depending on $\lambda$), and using the straightforward estimate
\begin{align*}
   \left(\frac{2\pi e}{t + \frac{2\pi\lambda}{\log t}}\right)^{+it+\tfrac{2\pi i\lambda}{\log t}}
&= \left(\frac{2\pi e}{t}\right)^{+it+\tfrac{2\pi i\lambda}{\log t}} \left(\frac{1}{1 + \frac{2\pi\lambda}{t \log t}}\right)^{+it+\tfrac{2\pi i\lambda}{\log t}} \\
&= \left(\frac{2\pi e}{t}\right)^{+it+\tfrac{2\pi i\lambda}{\log t}} \exp \left( - \frac{2\pi i\lambda}{\log t} \right) + \myo (t^{-1}) \,,
\end{align*}
it finally follows that
\begin{align*}
\big| \zeta(\tfrac{1}{2} + it + 2\pi i\lambda / \log t) \big|^2 
=
2 \re \! \big( S(\lambda,t) \big) + \myo(\log t) \,,
\end{align*}
where
\begin{align}
\label{sdefinition}
S(\lambda,t) := e^{\pi i/4} \left(\frac{2\pi e}{t}\right)^{it} \left(\frac{2\pi}{t}\right)^{\tfrac{2\pi i\lambda}{\log t}} \sum_{n \leq x(t)} d(n) \, n^{-\tfrac{1}{2}+it+\tfrac{2\pi i\lambda}{\log t}} \,.
\end{align}

Now suppose that we can show that
\begin{multline}
\
\lim_{T \to \infty} \frac{1}{T (\log T)^4} \int_{T}^{2T} 2\re (S(\mu,t)) \cdot 2\re (S(\nu,t)) \ dt \\
=
\frac{3/2}{\pi^4(\mu-\nu)^2} \left( 1 - \left( \frac{\sin \pi(\mu-\nu)}{\pi(\mu-\nu)} \right)^2 \right)
\
\label{approx41}
\end{multline}
for any $\mu,\nu \in \mathbb{R}$. \pagebreak[1]

\noindent{}It then follows by the Cauchy-Schwarz inequality that,
for $T$ sufficiently large,
\begin{align*}
&\mskip24mu 
   \int_{T}^{2T} \left| \zeta(\tfrac{1}{2} + i t + \tfrac{2 \pi i \mu}{\log t} ) \right|^2 \, \left| \zeta(\tfrac{1}{2} - i t - \tfrac{2 \pi i \nu}{\log t} ) \right|^2 \ dt \\
&= \int_{T}^{2T} 2\re (S(\mu,t)) \cdot 2\re (S(\nu,t)) \ dt \\
&\quad \,+\, \myo \left( \Bigg( \int_{T}^{2T} \Big| 2\re (S(\mu,t)) \Big|^2 \ dt \Bigg)^{1/2} \Bigg( \int_{T}^{2T} (\log t)^{2} \ dt \Bigg)^{1/2} \right) \\
&\quad \,+\, \myo \left( \Bigg( \int_{T}^{2T} (\log t)^{2} \ dt \Bigg)^{1/2} \Bigg( \int_{T}^{2T} \Big| 2\re (S(\nu,t)) \Big|^2 \ dt \Bigg)^{1/2} \right) \\
&\quad \,+\, \myo \left( \int_{T}^{2T} (\log t)^{2} \ dt \right) \\
&= \int_{T}^{2T} 2\re (S(\mu,t)) \cdot 2\re (S(\nu,t)) \ dt + o \Big( T \log^4 T \Big) \,,
\end{align*}
and the theorem is proved.

Thus, it remains to prove (\ref{approx41}).
To begin with, we have
$$
  2\re (S(\mu,t)) \cdot 2\re (S(\nu,t))
= 2\re (S(\mu,t) \, \overline{S(\nu,t)}) + 2\re (S(\mu,t) \, S(\nu,t))
$$
and therefore
\begin{multline}
\int_{T}^{2T} 2\re (S(\mu,t)) \cdot 2\re (S(\nu,t)) \ dt \\
= 2\re \left( \int_{T}^{2T} S(\mu,t) \, \overline{S(\nu,t)} \ dt \right)
+ 2\re \left( \int_{T}^{2T} S(\mu,t) \, S(\nu,t) \ dt \right) \,.
\label{approx42}
\end{multline}
An elaboration of the argument in \textsc{Ingham} \cite{In}
shows that the second integral on the right-hand side 
in~(\ref{approx42}) is of order $o(T \log^4 T)$ 
and therefore tends to zero after~division by $T \log^4 T$ 
as in (\ref{approx41}).
Indeed, from (\ref{sdefinition}), we have
$$
\int_{T}^{2T} S(\mu,t) \, S(\nu,t) \ dt
=
i \sum_{m,n \leq x(2T)} \frac{d(m) \, d(n)}{\sqrt{m} \, \sqrt{n}}
\,
\int_{T'}^{2T} \exp(iF(t)) \ dt \,,
$$
for all $T \geq 2$, where $T' := T'(T,m,n) := \max\,\{ T,2\pi m,2\pi n \}$ and
$$
F(t) := t \log(mn) + 2t \log(2\pi e / t) + \tfrac{2\pi\mu}{\log t} \log(2\pi m / t) + \tfrac{2\pi\nu}{\log t} \log(2\pi n / t) \,.
$$
The derivatives of $F(t)$ are given by
$$
F'(t) = 2 \log(2\pi\sqrt{mn} / t) - \left( 2\pi\mu \log(2 \pi m) + 2\pi\nu \log(2 \pi n) \right) \, \frac{1}{t (\log t)^2}
$$
and \mbox{\qquad\qquad}
$$
F''(t) = - 2/t + \left( 2\pi\mu \log(2 \pi m) + 2\pi\nu \log(2 \pi n) \right) \, \frac{2 + \log t}{t^2 (\log t)^3} \,.
$$
Clearly, for $T$ sufficiently large (depending on $\mu,\nu$),
$F''(t) \leq -1/t \leq -1/2T$ for~all $t \in [T',2T]$.
Hence, by Lemma 4.4 in \textsc{Titchmarsh} \cite{Ti}
(=~Lemma~\ref{bound2}), we~have the~estimate
$$
\left| \int_{T'}^{2T} \exp(iF(t)) \ dt \right| = \myo \big( \sqrt{T}\, \big) \,.
$$
Moreover, for $m \ne n$, $T' \geq 2\pi \max \,\{m,n\} > 2\pi \sqrt{mn}$.
It follows that, for $T$ sufficiently large (depending on $\mu,\nu$),
$F'(t) \leq \log(2\pi\sqrt{mn} / t) \leq \log(2\pi\sqrt{mn} / T')
\linebreak[1] \leq - \tfrac{1}{2} | \log (n/m) |$ for~all $t \geq T'$.
Hence, by Lemma 4.2 in \textsc{Titchmarsh} \cite{Ti}
(=~Lemma~\ref{bound1}), we~have the estimate
$$
\left| \int_{T'}^{2T} \exp(iF(t)) \ dt \right| = \myo \bigg( \frac{1}{|\log(n/m)|} \bigg) \,.
$$
\enlargethispage{1.0\baselineskip}%
Combining these two estimates and using
Lemmas B.3 and B.1 in \textsc{Ingham} \cite{In}
(=~Lemmas \ref{sum4} and \ref{d2dsum}), it follows that
\begin{align*}
   \int_{T}^{2T} & S(\mu,t) \, S(\nu,t) \ dt
 = i \sum_{m,n \leq x(2T)} \frac{d(m) \, d(n)}{\sqrt{m} \, \sqrt{n}} \,\cdot\, \int_{T'}^{2T} \exp(iF(t)) \ dt \\
&= \myo \left( \sum_{1 \le m < n \leq x(2T)} \frac{d(m) \, d(n)}{\sqrt{mn} \, \log(n/m)} \right) + \myo \left( \sum_{n \leq x(2T)} \frac{d(n)^2}{n} \, \sqrt{T} \right) \\
&= \myo \big( T \log^3 T \big) + \myo \big( \sqrt{T} \log^4 T \big)
 = o \big( T \log^4 T \big) \,,
\end{align*}
as claimed. \pagebreak[2]

Thus, it remains to examine the first integral 
on the right-hand side in (\ref{approx42}).
To~begin~with, from (\ref{sdefinition}), we have
\begin{align*}
   \int_{T}^{2T} & S(\mu,t) \, \overline{S(\nu,t)} \ dt \\
&= \sum_{m,n \leq x(2T)} \frac{d(m) \, d(n)}{\sqrt{m} \, \sqrt{n}} \int_{T'}^{2T} \left( \frac{2\pi}{t} \right)^{\tfrac{2\pi i(\mu-\nu)}{\log t}} \, m^{+it+\tfrac{2\pi i\mu}{\log t}} \, n^{-it-\tfrac{2\pi i\nu}{\log t}} \ dt \\
&= e^{-2\pi i(\mu-\nu)} \sum_{m,n \leq x(2T)} \frac{d(m) \, d(n)}{\sqrt{m} \, \sqrt{n}} \int_{T'}^{2T} (m/n)^{it} \left( 2\pi m \right)^{+\tfrac{2\pi i\mu}{\log t}} \, \left( 2\pi n \right)^{-\tfrac{2\pi i\nu}{\log t}} \ dt
\end{align*}
for all $T \geq 2$, where $T' := T'(T,m,n) := \max\,\{ T,2\pi m,2\pi n \}\,.$

Now, for those pairs $(m,n)$ with $m \ne n$,
we find using integration by parts that
\begin{align*}
&\mskip24mu \int_{T'}^{2T} (m/n)^{it} \left( 2\pi m \right)^{+\tfrac{2\pi i\mu}{\log t}} \, \left( 2\pi n \right)^{-\tfrac{2\pi i\nu}{\log t}} \ dt \\
&= \Biggl[ \frac{(m/n)^{it}}{i \log(m/n)} \left( 2\pi m \right)^{+\tfrac{2\pi i\mu}{\log t}} \left( 2\pi n \right)^{-\tfrac{2\pi i\nu}{\log t}} \Biggr]_{t=T'}^{t=2T} + \int_{T'}^{2T} \frac{(m/n)^{it}}{i \log(m/n)} \\ &\qquad\,\cdot\, \left( 2\pi m \right)^{+\tfrac{2\pi i\mu}{\log t}} \left( 2\pi n \right)^{-\tfrac{2\pi i\nu}{\log t}} \left[ \frac{2\pi i\mu \log(2\pi m)}{t (\log t)^2} - \frac{2\pi i\nu \log(2\pi n)}{t (\log t)^2} \right] \ dt \\
&= \myo \left( \frac{1}{|\log (m/n)|} \right) \,,
\end{align*}
where the last step uses the inequalities \pagebreak[2]
$T \leq T' \leq 2T$ and $m,n \leq x(2T) \leq T$.
Hence, using Lemma B.3 in \textsc{Ingham} \cite{In}
(= Lemma~\ref{sum4}), it follows that
\begin{multline*}
\sum_{m \ne n} \frac{d(m) \, d(n)}{\sqrt{m} \, \sqrt{n}} \int_{T'}^{2T} (m/n)^{it} \left( 2\pi m \right)^{+\tfrac{2\pi i\mu}{\log t}} \, \left( 2\pi n \right)^{-\tfrac{2\pi i\nu}{\log t}} \ dt \\
= \myo \left( \sum_{m \ne n} \frac{d(m) \, d(n)}{\sqrt{mn} \, |\log (m/n)|} \right)
= \myo \left( x(2T) \log^3 x(2T) \right)
= \myo \left( T \log^3 T \right) \,,
\end{multline*}
so that the sum over the pairs $(m,n)$ with $m \ne n$ tends to zero
after division by~$T \log^4 T$ as in (\ref{approx41}).

Consequently, to determine the asymptotic behaviour of the first integral
on the~right-hand side in (\ref{approx42}), it remains to consider
the sum over the pairs $(m,n)$ with $m = n$ and to show that
\begin{multline}
\lim_{T \to \infty} \frac{1}{T \log^4 T} \cdot 2\re \Bigg( e^{-2\pi i(\mu-\nu)} \, \sum_{n \leq x(2T)} \frac{d(n)^2}{n} \int_{T'}^{2T} \left( 2\pi n \right)^{\tfrac{2\pi i(\mu-\nu)}{\log t}} \ dt \Bigg)
\\ =
\frac{3/2}{\pi^4(\mu-\nu)^2} \left( 1 - \left( \frac{\sin \pi(\mu-\nu)}{\pi(\mu-\nu)} \right)^2 \right) \,.
\label{approx43}
\end{multline}
Clearly, in doing so, we may assume without loss of generality that $\nu = 0$.

To evaluate the integral on the left-hand side in (\ref{approx43}), write
$$
(2\pi n)^{2\pi i\mu/\log t} = (2\pi n)^{2\pi i\mu/\log T} - \int_{T}^{t} (2\pi n)^{2\pi i\mu/\log u} \, \log (2\pi n) \, \frac{2\pi i\mu}{u \, (\log u)^2} \ du
$$
and note that for $T \leq t \leq 2T$ and $n \leq 2T$,
$$
\int_{T}^{t} (2\pi n)^{2\pi i\mu/\log u} \, \log (2\pi n) \, \frac{2\pi i\mu}{u (\log u)^2} \ du
= \myo \left( (t-T) \cdot \frac{\log (2\pi n)}{T (\log T)^2} \right) \\
= \myo \left( \frac{1}{\log T} \right) \,.
$$
Hence, since $T' \geq T$ and $x(2T) \leq 2T$, it follows that
\begin{align*}
   \int_{T'}^{2T} (2\pi n)^{2\pi i\mu/\log t} \ dt
&=  \big( 2T - T' \big) \, (2\pi n)^{\tfrac{2\pi i\mu}{\log T}} + \myo \left( \frac{T}{\log T} \right) \nonumber\\
&= T \, (2\pi n)^{\tfrac{2\pi i\mu}{\log T}} - \big( T' - T \big) \, (2\pi n)^{\tfrac{2\pi i\mu}{\log T}} + \myo \left( \frac{T}{\log T} \right)
\end{align*}
and therefore
\begin{multline}
e^{-2\pi i\mu} \sum_{n \leq x(2T)} \frac{d(n)^2}{n} \int_{T'}^{2T} \left( 2\pi n \right)^{\tfrac{2\pi i\mu}{\log t}} \ dt
= e^{-2\pi i\mu} \sum_{n \leq x(2T)} T \, \frac{d(n)^2}{n} \, (2\pi n)^{\tfrac{2\pi i\mu}{\log T}}
\\
- e^{-2\pi i\mu} \sum_{n \leq x(2T)} \big( T' - T \big) \,\frac{d(n)^2}{n} \, (2\pi n)^{\tfrac{2\pi i\mu}{\log T}}
+ \myo \left( \frac{T}{\log T} \sum_{n \leq x(2T)} \frac{d(n)^2}{n} \right) \,.
\label{approx45}
\end{multline}
By Lemma B.1 in \textsc{Ingham} \cite{In} (= Lemma \ref{d2dsum}), we have
$$
\sum_{n \leq T} \frac{d(n)^2}{n} = \frac{1}{4\pi^2} \log^4 T + \myo(\log^3 T) \,.
$$
Thus, the $\myo$-term in (\ref{approx45}) is obviously of order $o(T \log^4 T)$.
Also, since $T' = T$ for~$n \leq x(T)$,
\begin{align*}
   \sum_{n \leq x(2T)} (T'- T) \, \frac{d(n)^2}{n}
&= \sum_{x(T) < n \leq x(2T)} (T'- T) \, \frac{d(n)^2}{n} \\
&\leq
   T \, \Biggl(\, \sum_{n \leq x(2T)} \frac{d(n)^2}{n} - \sum_{n \leq x(T)} \frac{d(n)^2}{n} \Biggr) \\
&= T \, \Biggl( \frac{1}{4\pi^2} \log^4 (2T) - \frac{1}{4\pi^2} \log^4 (T) + \myo ( \log^3 T ) \Biggr) \\
&= o(T \log^4 T) \,,
\end{align*}
so that the second sum on the right-hand side in (\ref{approx45})
is also of order $o(T \log^4 T)$.

Thus, it remains to consider
the first sum on the right-hand side in (\ref{approx45}).
Similarly as in the proof of Lemma B.1 in~\textsc{Ingham}~\cite{In},
we can approximate this~sum by an integral. Setting
$$
D(t) := \sum_{n \leq t} d(n)^2
$$
and using Lemma \ref{d2sum}, we have,
for $\lambda \in \mathbb{R}$ from a bounded set,
\begin{align*}
&\mskip24mu \sum_{n \leq x(2T)} d(n)^2 \, n^{-1+i\lambda} \\
&= \sum_{n \leq x(2T)} \left( D(n) - D(n-1) \right) \, n^{-1+i\lambda} \\
&= \sum_{n \leq x(2T)-1} D(n) \left( n^{-1+i\lambda} - (n+1)^{-1+i\lambda} \right) + \myo(\log^3 T) \allowdisplaybreaks\\
&= (1-i\lambda) \int_{1}^{x(2T)} \frac{D(u)}{u^{2-i\lambda}} \ du + \myo(\log^3 T) \\
&= (1-i\lambda) \, \frac{1}{\pi^2} \int_{1}^{x(2T)} \frac{\log^3 u}{u^{1-i\lambda}} \ du + \myo\bigg(\int_1^{x(2T)} \frac{\log^2 u}{u} \ du \bigg) + \myo(\log^3 T) \\
&= (1-i\lambda) \, \frac{1}{\pi^2} \int_{1}^{x(2T)} \frac{\log^3 u}{u^{1-i\lambda}} \ du + \myo(\log^3 T) \,.
\end{align*}
Substituting $v = \log u$ and $w = v / \log T$ yields
\begin{align*}
   \int_{1}^{x(2T)} \frac{\log^3 u}{u^{1-i\lambda}} \ du
&= \int_{0}^{\log x(2T)} v^3 \, e^{i\lambda v} \ dv \\
&= (\log T)^4 \, \int_{0}^{\log x(2T) / \log T} w^3 \, e^{i\lambda w \log T} \ dw
\end{align*}
and therefore
$$
\sum_{n \leq x(2T)} d(n)^2 \, n^{-1+i\lambda} \\
=
(\log T)^4 \, (1-i\lambda) \cdot \frac{1}{\pi^2} \int_{0}^{1} w^3 \, e^{i\lambda w \log T} \ dw + \myo(\log^3 T) \,.
$$
Thus, with $\lambda$ replaced by $2\pi\mu/\log T$, it follows that
\begin{align*}
&\mskip24mu
   T \, e^{-2\pi i\mu} \sum_{n \leq x(2T)} \frac{d(n)^2}{n} \, (2\pi n)^{\tfrac{2\pi i\mu}{\log T}} \\ 
&= T \, (\log T)^4 \, (2\pi)^{\tfrac{2\pi i\mu}{\log T}} \, (1-\tfrac{2\pi i\mu}{\log T}) \cdot \frac{1}{\pi^2} \int_{0}^{1} w^3 \, e^{2\pi i\mu(w-1)} \ dw + \myo(T \log^3 T) \\
&= T \, (\log T)^4 \cdot \frac{1}{\pi^2} \int_{0}^{1} w^3 \, e^{2\pi i\mu(w-1)} \ dw + \myo(T \log^3 T) \,.
\end{align*}
Dividing by $T \, \log^4 T$ and taking real parts, we therefore obtain
\begin{multline*}
  \lim_{T \to \infty} \frac{1}{T \log^4 T} \cdot 2 \re \Bigg( T \, e^{-2\pi i\mu} \sum_{n \leq x(2T)} \frac{d(n)^2}{n} \, (2\pi n)^{\tfrac{2\pi i\mu}{\log T}} \Bigg) \\
= \frac{2}{\pi^2} \int_{0}^{1} w^3 \, \cos(2\pi\mu(w-1)) \ dw \,.
\end{multline*}
A direct calculation using the trigonometric identity
$1 - \cos(z) = 2 \sin^2 (z/2)$ yields
\begin{align*}
   \int_{0}^{1} w^3 \, \cos(2\pi\mu(w-1)) \ dw
&= \frac{3}{(2\pi\mu)^2} \left( 1 - \frac{2 - 2 \cos(2\pi\mu)}{(2\pi\mu)^2} \right) \\
&= \frac{3}{(2\pi\mu)^2} \left( 1 - \left( \frac{\sin \pi\mu}{\pi\mu} \right)^2 \right) \,.
\end{align*}
This is true also for $\mu = 0$, provided that we consider
the continuous extension of the right-hand side, \ie 1/4. 
This concludes the proof of~(\ref{approx43}),
and hence of Theorem \ref{zeta4theorem}.
\end{proof}

\bigskip

\section{The Conjecture for the Higher Shifted Moments}

In this section we comment on the relationship between Conjecture~\ref{skc}
for the~higher (even) shifted moments of the Riemann zeta function
and the conjecture by \textsc{Conrey}, \textsc{Farmer}, \textsc{Keating}, 
\textsc{Rubinstein}, and \textsc{Snaith} \pagebreak[2] \cite{CFKRS1,CFKRS2},
which we will simply call the CFKRS-Conjecture from now on.
 
In the special case of the Riemann zeta function,
this conjecture can be stated as follows:

\begin{conjecture}[Conjecture 2.2 in \cite{CFKRS1}]
\label{cfkrs}
For any $M=1,2,3,\hdots,$ and any~$\mu_1,\hdots,\mu_M,\nu_1,\hdots,\nu_M \in \mathbb{R}$,
\begin{multline*}
\int_{0}^{T} 
\prod_{j=1}^{M} \zeta \left(\tfrac{1}{2} + it + i\mu_j \right)
\prod_{j=1}^{M} \zeta \left(\tfrac{1}{2} - it - i\nu_j \right)
\ dt
\\ =
\int_{0}^{T}
W_M(t;i\mu_1,\hdots,i\mu_M;i\nu_{1},\hdots,i\nu_M)
\left( 1 + \myo(t^{-(1/2) + \varepsilon}) \right) 
\ dt \,, 
\end{multline*}
where
\begin{multline*}
W_M(t;\xi_1,\hdots,\xi_M,\xi_{M+1},\hdots,\xi_{2M}) \\
:= \exp (\tfrac{1}{2} \log \tfrac{t}{2\pi} \cdot {\textstyle\sum\limits_{j=1}^{M} (-\xi_j+\xi_{M+j})})
\,\cdot\, \sum_{\sigma \in \mys_{2M}'} \exp (\tfrac{1}{2} \log \tfrac{t}{2\pi} \cdot {\textstyle\sum\limits_{j=1}^{M} (\xi_{\sigma(j)}-\xi_{\sigma(M+j)})})
\\ \,\cdot\, A_M(\xi_{\sigma(1)},\hdots,\xi_{\sigma(2M)}) \cdot \prod_{j,k=1,\hdots,M} \zeta(1 + \xi_{\sigma(j)} - \xi_{\sigma(M+k)}) \,. 
\end{multline*}
Here, $\mys_{2M}'$ denotes the subset of permutations $\sigma$
of the set $\{ 1,\hdots,2M \}$ satisfying 
$\sigma(1) < \cdots < \sigma(M)$ and 
$\sigma(M+1) \linebreak[1] < \cdots < \linebreak[1] \sigma(2M)$,
and $A_M(z_1,\hdots,z_{2M})$ is a~certain function
which is analytic in a~neighborhood of the origin
and for which $A_M(0,\hdots,0) = a_M$.
\end{conjecture}

We will show that Conjecture \ref{skc} follows from the CFKRS-conjecture
provided that one permits replacing $\mu_1,\hdots,\mu_M,\nu_1,\hdots,\nu_M$ with 
$2\pi \mu_1/\log t,\hdots,2\pi \mu_M/\log t, \linebreak[2] 2\pi \nu_1/\log t,\hdots,2\pi \nu_M/\log t$.
In this respect, Conjecture \ref{skc} may be regarded as a~special~case
of the CFKRS-conjecture. 

Similarly as in the proof of Theorem \ref{zeta4theorem}, 
we prefer working with the interval $[T,2T]$ instead of $[0,T]$.
Besides that, we will only consider the leading-order terms.
We then have the approximation
\begin{multline}
\label{zetaf2}
\int_{T}^{2T} 
\prod_{j=1}^{M} \zeta \left(\tfrac{1}{2} + it + \tfrac{2\pi i\mu_j}{\log t} \right)
\prod_{j=1}^{M} \zeta \left(\tfrac{1}{2} - it - \tfrac{2\pi i\nu_j}{\log t} \right)
\ dt
\\ \approx
\int_{T}^{2T}
  \exp (\tfrac{1}{2} \log \tfrac{t}{2\pi} \cdot {\textstyle\sum_{j=1}\limits^{M} (-\frac{\xi_j}{\log t}+\frac{\xi_{M+j}}{\log t})})
\,\cdot\, \sum_{\sigma \in \mys_{2M}'} \exp (\tfrac{1}{2} \log \tfrac{t}{2\pi} \cdot {\textstyle\sum\limits_{j=1}^{M} (\frac{\xi_{\sigma(j)}}{\log t}-\frac{\xi_{\sigma(M+j)}}{\log t})})
\\ \,\cdot\, A_M(\tfrac{\xi_{\sigma(1)}}{\log t},\hdots,\tfrac{\xi_{\sigma(2M)}}{\log t}) \cdot \prod_{j,k=1,\hdots,M} \zeta(1 + \tfrac{\xi_{\sigma(j)}}{\log t} - \tfrac{\xi_{\sigma(M+k)}}{\log t})
\ dt \,, 
\end{multline}
where we have put 
$\xi_j := 2 \pi i \mu_j$ for $j=1,\hdots,M$, 
$\xi_{M+j} := 2 \pi i \nu_j$ for $j=1,\hdots,M$,
and $\mathcal{S}_{2M}'$ and $A_M$ are the same as in the CFKRS-conjecture.
Alternatively, \linebreak the approximation (\ref{zetaf2})
could be obtained by starting from the expression
$$
\int_{T}^{2T} \prod_{j=1}^{M} \zeta(\tfrac{1}{2} + it + \tfrac{2\pi i\mu_j}{\log t}) \, \prod_{j=1}^{M} \zeta(\tfrac{1}{2} - it - \tfrac{2\pi i\nu_j}{\log t}) \ dt
$$
and by following the (non-rigorous) ``recipe'' leading to the CFKRS-conjecture.
(In~fact, since the factor $\frac{1}{\log t}$ is essentially constant,
it is irrelevant for the question which terms are rapidly oscillating
and should therefore be discarded.) 

To simplify (\ref{zetaf2}) as $T \to \infty$,
recall that $A_M$ is regular at~$(0,\hdots,0)$ 
and $\zeta$ has a~simple pole with residual $1$ at $z = 1$.
Thus, concentrating on leading-order terms, we obtain
\begin{multline}
\label{zetaf3}
\int_{T}^{2T} 
\prod_{j=1}^{M} \zeta \left(\tfrac{1}{2} + it + \tfrac{2\pi i\mu_j}{\log t} \right)
\prod_{j=1}^{M} \zeta \left(\tfrac{1}{2} - it - \tfrac{2\pi i\nu_j}{\log t} \right)
\ dt
\\ \approx
\int_{T}^{2T}
  \exp (\tfrac{1}{2} \cdot {\textstyle\sum_{j=1}\limits^{M} (-\xi_j+\xi_{M+j})}
\,\cdot\, \sum_{\sigma \in \mys_{2M}'} \exp (\tfrac{1}{2} \cdot {\textstyle\sum\limits_{j=1}^{M} (\xi_{\sigma(j)}-\xi_{\sigma(M+j)})})
\\ \,\cdot\, A_M(0,\hdots,0) \cdot \frac{(\log t)^{M^2}}{\prod_{j,k=1,\hdots,M} (\xi_{\sigma(j)} - \xi_{\sigma(M+k)})}
\ dt \,.
\end{multline}
Therefore, since
$$
  \int_{T}^{2T} (\log t)^{M^2} \ dt
= T (\log T)^{M^2} + \myo \left( T (\log T)^{M^2 - 1} \right) \,,
$$
we should expect that
\begin{multline*}
\lim_{T \to \infty}
\frac{1}{T (\log T)^{M^2}} 
\int_{T}^{2T} 
\prod_{j=1}^{M} \zeta(\tfrac{1}{2} + it + \tfrac{2\pi i\mu_j}{\log t})
\prod_{j=1}^{M} \zeta(\tfrac{1}{2} - it - \tfrac{2\pi i\nu_j}{\log t})
\ dt
\\ = \exp (\tfrac{1}{2} \cdot {\textstyle\sum_{j=1}^{M} (-\xi_j+\xi_{M+j})})
\cdot \sum_{\sigma \in \mys_{2M}'} \exp (\tfrac{1}{2} \cdot {\textstyle\sum_{j=1}^{M} (\xi_{\sigma(j)}-\xi_{\sigma(M+j)})}) \qquad
\\ \,\cdot\, A_M(0,\hdots,0) \cdot \frac{1}{\prod_{j,k=1,\hdots,M} (\xi_{\sigma(j)} - \xi_{\sigma(M+k)})} \,. 
\end{multline*}
Since
$
A_M(0,\hdots,0) = a_M
$ 
(see Equation (2.7.10) in \cite{CFKRS2})
and
\begin{multline*}
\sum_{\sigma \in \mys_{2M}'} \exp (\tfrac{1}{2} \cdot {\textstyle\sum_{j=1}^{M} (\xi_{\sigma(j)}-\xi_{\sigma(M+j)})})
\cdot \frac{1}{\prod_{j,k=1,\hdots,M} (\xi_{\sigma(j)} - \xi_{\sigma(M+k)})}
\\
= \frac{1}{\Delta(2\pi\mu_1,\hdots,2\pi\mu_M) \cdot {\Delta(2\pi\nu_1,\hdots,2\pi\nu_M)}} \cdot \det \left( \frac{\sin \pi(\mu_j-\nu_k)}{\pi(\mu_j-\nu_k)} \right)_{j,k=1,\hdots,M}
\end{multline*}
(see equation (\ref{cue6}) in Appendix~B),
this yields Conjecture \ref{skc}.

\bigskip

\appendix 

\section{Some Estimates from the Literature}

In this appendix, we state some estimates from the literature
which have been used in the proofs
of Theorems \ref{zeta2theorem} and \ref{zeta4theorem}.

\begin{lemma}[Titchmarsh \cite{Ti}, Lemma 4.2]
\label{bound1}
Let $F(x)$ be a real differentiable function such that $F'(x)$ is monotonic,
and $F'(x) \geq \varepsilon > 0$ or $F'(x) \leq -\varepsilon < 0$, 
throughout the interval $[a,b]$. Then
$$
\left| \int_{a}^{b} \exp(iF(x)) \ dx \right| \leq \frac{4}{\varepsilon} \,.
$$
\end{lemma}

\begin{lemma}[Titchmarsh \cite{Ti}, Lemma 4.4]
\label{bound2}
Let $F(x)$ be a real function, twice differentiable, such that
$F''(x) \geq \varepsilon > 0$ or $F''(x) \leq -\varepsilon < 0$, 
throughout the interval $[a,b]$. Then
$$
\left| \int_{a}^{b} \exp(iF(x)) \ dx \right| \leq \frac{8}{\sqrt{\varepsilon}} \,.
$$
\end{lemma}

The $\myo$-bounds in the following lemmas relate to the case that $T \to \infty$.

\begin{lemma}[Titchmarsh \cite{Ti}, Lemma 7.2]
\label{sum2}
$$
\sum_{1 \le m < n \le T} \frac{1}{\sqrt{mn} \, \log (n/m)} = \myo \left( T \, \log T \right) \,.
$$
\end{lemma}

\begin{lemma}[Ingham \cite{In}, Lemma B.3]
\label{sum4}
$$
\sum_{1 \le m < n \le T} \frac{d(m) \, d(n)}{\sqrt{mn} \, \log (n/m)} = \myo \left( T \, \log^3 T \right) \,.
$$
\end{lemma}

\begin{lemma}[see e.\,g. Ivi\'c \cite{Iv}, Equation (5.24)]
\label{d2sum}
$$
\sum_{n \leq T} d(n)^2 = \frac{1}{\pi^2} \, T \, \log^3 T + \myo \left( T \, \log^2 T \right) \,.
$$
\end{lemma}

\begin{lemma}[Ingham \cite{In}, Lemma B.1]
\label{d2dsum}
$$
\sum_{n \leq T} \frac{d(n)^2}{n} = \frac{1}{4 \pi^2} \log^4 T + \myo(\log^3 T) \,.
$$
\end{lemma}

\bigskip

\section{On the Characteristic Polynomial of the CUE}

The purpose of this appendix is to show that
\begin{multline}
\label{cue4}
\lim_{N \to \infty} \frac{1}{N^{M^2}} \cdot \fc \left( N; e^{2\pi i\mu_1/N},\hdots,e^{2\pi i\mu_M/N}, e^{2\pi i\nu_1/N},\hdots,e^{2\pi i\nu_M/N} \right) 
\\ = 
\frac{\exp(\sum_{j=1}^{M} \pi i(\mu_j-\nu_j))}{\Delta(2\pi\mu_1,\hdots,2\pi\mu_M) \cdot {\Delta(2\pi\nu_1,\hdots,2\pi\nu_M)}} \cdot \det \left( \frac{\sin \pi(\mu_j-\nu_k)}{\pi(\mu_j-\nu_k)} \right) 
\end{multline}
and
\begin{multline}
\label{cue5}
\lim_{N \to \infty} \frac{1}{N^{M^2}} \cdot \fc \left( N; e^{2\pi i\mu_1/N},\hdots,e^{2\pi i\mu_M/N}, e^{2\pi i\nu_1/N},\hdots,e^{2\pi i\nu_M/N} \right) 
\\ = 
\exp(\tfrac{1}{2} \sum_{j=1}^{M} (\xi_j-\xi_{M+j}))
\cdot
\sum_{\sigma \in \mys_{2M}'} \frac{\exp (\tfrac{1}{2} {\textstyle\sum_{j=1}^{M} (\xi_{\sigma(j)}-\xi_{\sigma(M+j)})})}{\prod_{j,k=1,\hdots,M} (\xi_{\sigma(j)} - \xi_{\sigma(M+k)})} \,,
\end{multline}
where $\Delta(x_1,\hdots,x_M) := \prod_{j < k} (x_k - x_j)$
denotes the Vandermonde determinant,
$\mys_{2M}'$ denotes the subset of permutations $\sigma$
of the set $\{ 1,\hdots,2M \}$ satisfying \linebreak
 $\sigma(1) < \cdots < \sigma(M)$ and $\sigma(M+1) < \cdots < \sigma(2M)$,
$\xi_j := 2 \pi i \mu_j$ for $j=1,\hdots,M$, and
$\xi_{M+j} := 2 \pi i \nu_j$ for $j=1,\hdots,M$.
In particular, by~combining (\ref{cue4}) and (\ref{cue5}),
it follows that
\begin{multline}
\label{cue6}
\sum_{\sigma \in \mys_{2M}'} \frac{\exp (\tfrac{1}{2} {\textstyle\sum_{j=1}^{M} (\xi_{\sigma(j)}-\xi_{\sigma(M+j)})})}{\prod_{j,k=1,\hdots,M} (\xi_{\sigma(j)} - \xi_{\sigma(M+k)})}
\\
= \frac{1}{\Delta(2\pi\mu_1,\hdots,2\pi\mu_M) \cdot {\Delta(2\pi\nu_1,\hdots,2\pi\nu_M)}} \cdot \det \left( \frac{\sin \pi(\mu_j-\nu_k)}{\pi(\mu_j-\nu_k)} \right) \,, 
\end{multline}
which was used at the end of Section~3. \pagebreak[2]

The proofs of (\ref{cue4}) and (\ref{cue5}) use well-known arguments from 
random matrix theory, and are included here mainly for the sake of completeness. 

To prove (\ref{cue4}), we use an argument
from Section 4.1 in \textsc{Forrester} \cite{Fo}.
Recall that the correlation function of order $2M$
of the characteristic poly\-nomial of a~random matrix
from the Circular Unitary Ensemble is defined by
$$
  f(\mu_1,\hdots,\mu_M;\nu_1,\hdots,\nu_M) \\
= \int_{\mathcal{U}_N} \, \prod_{j=1}^{M} \det(U - \mu_j I) \, \overline{\det(U - \nu_j I)} \, \ dU \,.
$$
It is well-known that the probability measure on the space 
of eigenvalue angles induced by the CUE is given by
$$
Z_N^{-1} \prod_{1 \leq j<k \leq N} \left| e^{i\vartheta_k} - e^{i\vartheta_j} \right|^2 \ d\mylebesgue^N(\vartheta_1,\hdots,\vartheta_N)
$$
(see \textsc{Forrester}~\cite{Fo} or \textsc{Mehta}~\cite{Me}),
where $Z_N := (2\pi)^N \, N!$ and $\mylebesgue$ denotes the~Lebesgue measure 
on the interval $[0,2\pi]$. We therefore obtain
\begin{align*}
&\mskip24mu f(e^{i\mu_1},\hdots,e^{i\mu_M};e^{i\nu_1},\hdots,e^{i\nu_M}) \\
&= Z_N^{-1} \int \prod_{j=1}^{M} \, \prod_{k=1}^{N} (e^{i\vartheta_k} - e^{i\mu_j}) \, \prod_{j=1}^{M} \prod_{k=1}^{N} \overline{(e^{i\vartheta_k}- e^{i\nu_j})} 
\\ &\mskip250mu  \,\cdot\, \prod_{1 \leq j<k \leq N} \left| e^{i\vartheta_k} - e^{i\vartheta_j} \right|^2 \ d\mylebesgue^N(\vartheta_1,\hdots,\vartheta_N) \\
&= \frac{Z_N^{-1}}{C(\mu,\nu)} \int \Delta(e^{i\mu_1},\hdots,e^{i\mu_M},e^{i\vartheta_1},\hdots,e^{i\vartheta_N}) 
\\[+5pt] &\mskip135mu \,\cdot\, \Delta(e^{-i\nu_1},\hdots,e^{-i\nu_M},e^{-i\vartheta_1},\hdots,e^{-i\vartheta_N}) \ d\mylebesgue^N(\vartheta_1,\hdots,\vartheta_N) \\[+5pt]
&= \frac{Z_N^{-1}}{C(\mu,\nu)} \int \det \left( \begin{array}{c} e^{ik\mu_j} \\ \hline \\[-10pt] e^{ik\vartheta_j} \end{array} \right)_{jk} 
\,\cdot\, \det \left( \begin{array}{c|c} e^{-ik\nu_l} & e^{-ik\vartheta_l} \end{array} \right)_{kl} \ d\mylebesgue^N(\vartheta_1,\hdots,\vartheta_N) \\[+5pt]
&= \frac{Z_N^{-1}}{C(\mu,\nu)} \int \det \left( \begin{array}{c|c} S_{N+M}(\mu_j,\nu_l) & S_{N+M}(\mu_j,\vartheta_l) \\[+1pt] \hline \\[-11pt] S_{N+M}(\vartheta_j,\nu_l) & S_{N+M}(\vartheta_j,\vartheta_l) \end{array} \right)_{jl}
\ d\mylebesgue^N(\vartheta_1,\hdots,\vartheta_N) \,,
\end{align*}
where
$
\Delta(x_1,\hdots,x_n) := \prod_{1 \leq j<k \leq n} (x_k - x_j)
$
denotes the Vandermonde determinant,
$$
C(\mu,\nu) := \Delta(e^{i\mu_1},\hdots,e^{i\mu_M}) \,\cdot\, \Delta(e^{-i\nu_1},\hdots,e^{-i\nu_M}) \,,
$$
and
$$
   S_{n}(\mu,\nu) 
:= \sum_{k=0}^{n-1} e^{ik(\mu-\nu)}
 = \frac{e^{in(\mu-\nu)} - 1}{e^{i(\mu-\nu)} - 1}
 = e^{i(n-1)(\mu-\nu)/2} \cdot \frac{\sin(n(\mu-\nu)/2)}{\sin((\mu-\nu)/2)} \,.
$$
Carrying out the integration with respect to $\vartheta_N,\hdots,\vartheta_1$ 
as in the proof of Pro\-position 4.2 in \textsc{Forrester} \cite{Fo},
it follows that
\begin{align*}
f(e^{i\mu_1},\hdots,e^{i\mu_M};e^{i\nu_1},\hdots,e^{i\nu_M})
= 
\frac{1}{C(\mu,\nu)} \cdot \det \Big( S_{N+M}(\mu_j,\nu_l) \Big)_{jl} \,.
\end{align*}

Replacing $e^{i\mu_j}$, $e^{i\nu_j}$ with $e^{2\pi i\mu_j/N}$, $e^{2\pi i\nu_j/N}$,
multiplying by $N^{-M^2}$ and letting \mbox{$N \to \infty$}, we therefore obtain
\begin{align*}
&\mskip24mu \lim_{N \to \infty} \Big( N^{-M^2} \, f(e^{2\pi i\mu_1/N},\hdots,e^{2\pi i\mu_M/N};e^{2\pi i\nu_1/N},\hdots,e^{2\pi i\nu_M/N}) \Big) \\
&= \lim_{N \to \infty} \frac{\exp(\sum_{j=1}^{M} \pi i(N+M-1)(\mu_j-\nu_j)/N)}{\Delta(Ne^{2\pi i\mu_1/N},\hdots,Ne^{2\pi i\mu_M/N}) \, \Delta(Ne^{-2\pi i\nu_1/N},\hdots,Ne^{-2\pi i\nu_M/N})}
\\&\mskip300mu \,\cdot\, \det \bigg( \frac{\sin(\pi (N+M)(\mu_j-\nu_l)/N)}{N \sin(\pi (\mu_j-\nu_l)/N)} \bigg) \\
&= \frac{\exp(\sum_{j=1}^{M} \pi i(\mu_j-\nu_j))}{\Delta(2\pi\mu_1,\hdots,2\pi\mu_M) \, \Delta(2\pi\nu_1,\hdots,2\pi\nu_M)} \cdot \det \bigg( \frac{\sin \pi(\mu_j-\nu_l)}{\pi (\mu_j-\nu_l)} \bigg) \,,
\end{align*}
and (\ref{cue4}) is proved.

\pagebreak[1]

To prove (\ref{cue5}), we use the representation
\begin{align*}
& \mskip24mu f ( e^{2\pi i\mu_1},\hdots,e^{2\pi i\mu_M};e^{2\pi i\nu_1},\hdots,e^{2\pi i\nu_M}) \\
&= \exp \big( \tfrac{1}{2}N \sum_{j=1}^{M} (\xi_{j} - \xi_{M+j}) \big) \cdot \sum_{\sigma \in \mys_{2M}'} \frac{\exp \big( \tfrac{1}{2}N \sum_{j=1}^{M} (\xi_{\sigma(j)} - \xi_{\sigma(M+j)}) \big)}{\prod_{j,k=1,\hdots,M} \big( 1 - e^{\xi_{\sigma(M+k)} - \xi_{\sigma(j)}} \big)} \,,
\end{align*}
where $\mathcal{S}_{2M}'$ and $\xi_j$ are defined as below (\ref{cue5}).
See Equation (2.21) in \textsc{Conrey}, \textsc{Farmer}, \textsc{Keating}, 
\textsc{Rubinstein}, and \textsc{Snaith} \cite{CFKRS1}, but note that 
we use a slightly different definition of the characteristic polynomial,
which explains why some signs have changed.

Replacing $e^{2\pi i\mu_j}$, $e^{2\pi i\nu_j}$ with $e^{2\pi i\mu_j/N}$, $e^{2\pi i\nu_j/N}$,
multiplying by $N^{-M^2}$ and letting \mbox{$N \to \infty$}, it follows that
\begin{align*}
&\mskip24mu \lim_{N \to \infty} \Big( N^{-M^2} \, f(e^{2\pi i\mu_1/N},\hdots,e^{2\pi i\mu_M/N};e^{2\pi i\nu_1/N},\hdots,e^{2\pi i\nu_M/N}) \Big) \\
&= \exp \big( \tfrac{1}{2} \sum_{j=1}^{M} (\xi_{j} - \xi_{M+j}) \big) \cdot \sum_{\sigma \in \mys_{2M}'} \frac{\exp \big( \tfrac{1}{2} \sum_{j=1}^{M} (\xi_{\sigma(j)} - \xi_{\sigma(M+j)}) \big)}{\prod_{j,k=1,\hdots,M} \lim\limits_{N \to \infty} \big( N \cdot \big( 1 - e^{(\xi_{\sigma(M+k)} - \xi_{\sigma(j)})/N} ) \big)} \\
&= \exp \big( \tfrac{1}{2} \sum_{j=1}^{M} (\xi_{j} - \xi_{M+j}) \big) \cdot \sum_{\sigma \in \mys_{2M}'} \frac{\exp \big( \tfrac{1}{2} \sum_{j=1}^{M} (\xi_{\sigma(j)} - \xi_{\sigma(M+j)}) \big)}{\prod_{j,k=1,\hdots,M} \big( \xi_{\sigma(j)} - \xi_{\sigma(M+k)} \big)} \,,
\end{align*}
and (\ref{cue5}) is proved.

\bigskip

\section*{Acknowledgements}

I thank Friedrich G\"otze for several valuable discussions.
Furthermore, I thank Steve Gonek as well as an anonymous referee 
for pointing out relevant references from the classical area 
of analytic number theory.

\bigskip

\renewcommand{\refname}{REFERENCES}

\bigskip


\begin{thebibliography}{ABCDEF}

	\bibitem[At]{At}
	Atkinson, F.\,V. (1948):
	A mean value property of the Riemann zeta function.
	\textit{J. London Math. Soc.}, \textbf{23}, 128--135.

 	\bibitem[BS]{BS}
 	Borodin, A.; Strahov, E. (2006):
 	Averages of characteristic poly\-nomials in random matrix theory.
 	\textit{Comm. Pure Appl. Math.}, \textbf{59}, \mbox{161--253}.

    \bibitem[BH1]{BH1}
    Br\'ezin, E.; Hikami, S. (2000):
	Characteristic polynomials of random matrices.
	\textit{Comm. Math. Phys.}, \textbf{214}, 111--135.

	\bibitem[BH2]{BH2}
	Br\'ezin, E.; Hikami, S. (2001):
	Characteristic polynomials of real symmetric random matrices.
	\textit{Comm. Math. Phys.}, \textbf{223}, 363--382.

	\bibitem[CFKRS1]{CFKRS1}
	Conrey, J.B.; Farmer, D.W.; Keating, J.P.; Rubinstein, M.O.; Snaith, N.C. (2003):
	Autocorrelation of Random Matrix Poly\-nomials.
	\textit{Comm. Math. Phys.}, \textbf{237}, 365--395.

	\bibitem[CFKRS2]{CFKRS2}
	Conrey, J.B.; Farmer, D.W.; Keating, J.P.; Rubinstein, M.O.; Snaith, N.C. (2005):
	Integral Moments of $L$-functions.
	\textit{Proc. London Math. Soc.}, \textbf{91}, 33--104.

	\bibitem[CFKRS3]{CFKRS3}
	Conrey, J.B.; Farmer, D.W.; Keating, J.P.; Rubinstein, M.O.; Snaith, N.C. (2007):
	Lower order terms in the full moment conjecture for the Riemann zeta function.
	Preprint.

    \bibitem[Fo]{Fo}
    Forrester, P.J. (2007+):
    \textit{Log Gases and Random Matrices}.
    Preprint, \verb|www.ms.unimelb.edu.au/~matpjf/matpjf.html|

    \bibitem[Go]{Go}
    Gonek, S.M. (1984):
	Mean values of the Riemann zeta function and its derivatives.
	\textit{Invent. Math.}, \textbf{75}, 123--141.

	\bibitem[GK]{GK}
	G\"otze, F.; K\"osters, H. (2008):
    On the second-order correlation function of the characteristic polynomial of a Hermitian random matrix.
	To appear in \textit{Comm. Math. Phys.}


    \bibitem[Hu]{Hu}
    Hughes, C.P. (2003):
	Random matrix theory and discrete moments of the Riemann zeta function.
	\textit{J. Phys. A.}, \textbf{36}, 2907--2917.

	\bibitem[In]{In}
	Ingham, A.E. (1926):
	Mean value theorems in the theory of the Riemann zeta function.
	\textit{Proc. London Math. Soc.}, \textbf{27}, 273--300.

	\bibitem[Iv]{Iv}
	Ivi\'c, A. (1985):
	\textit{The Riemann zeta function.}
	John Wiley \& Sons, New York.

    \bibitem[KS1]{KS1}
    Keating, J.P.; Snaith, N.C. (2000):
    Random matrix theory and $\zeta(1/2 + it)$.
    \textit{Comm. Math. Phys.}, \textbf{214}, 57--89.

    \bibitem[KS2]{KS2}
    Keating, J.P.; Snaith, N.C. (2000):
    Random matrix theory and $L$-functions at $s = 1/2$.
    \textit{Comm. Math. Phys.}, \textbf{214}, 91--110.

    \bibitem[Me]{Me}
    Mehta, M.L. (2004):
    \textit{Random Matrices}, 3rd edition.
    Pure and Applied Mathematics, vol. 142, Elsevier, Amsterdam.

    \bibitem[MN]{MN}
    Mehta, M.L.; Normand, J.-M. (2001):
    Moments of the characteristic polynomial in the three ensembles of random matrices.  
    \textit{J. Phys. A}, \textbf{34}, 4627--4639.

	\bibitem[Mon]{Mon}
	Montgomery, H.L. (1973):
	The pair correlation of zeros of the Riemann zeta function.
	\textit{Proc. Symp. Pure Math.}, \textbf{24}, 181--193.

    \bibitem[Mot]{Mot}
    Motohashi, Y. (1997):
    \emph{Spectral theory of the Riemann $\zeta$-function.}
    Cambridge Tracts in Mathematics, vol. 127, Cambridge University Press, Cambridge.
    
    \bibitem[Moz1]{Moz1}
    Mozer, J. (1980):
    Proof of E. C. Titchmarsh's conjecture in the theory of the Riemann zeta function.
    \textit{Acta Arith.}, \textbf{36}, 147--156.

    \bibitem[Moz2]{Moz2}
    Mozer, J. (1983):
    A biquadratic sum in the theory of the Riemann zeta function.
    \textit{Acta Math. Univ. Comenian.}, \textbf{42/43}, 35--39.

    \bibitem[Moz3]{Moz3}
    Mozer, J. (1991):
    On the order of a Titchmarsh sum in the theory of the Riemann zeta function.
    \textit{Czechoslovak. Math. J.}, \textbf{41}, 663--684.

    \bibitem[SF]{SF3}
    Strahov, E.; Fyodorov, Y.V. (2003):
    Universal results for correlations of characteristic polynomials:
        Riemann-Hilbert approach.
    \textit{Comm. Math. Phys.}, \textbf{241}, 343--382.

	\bibitem[Ti]{Ti}
    Titchmarsh, E.C. (1986):
    \textit{The Theory of the Riemann zeta function}, 2nd edition.
	Oxford University Press, Oxford.

\end{thebibliography}
\end{document}